\documentclass[a4paper,reqno,12pt]{amsart}

\usepackage[utf8]{inputenc}
\usepackage[english]{babel}
\usepackage{amssymb}
\usepackage{amsfonts}
\usepackage{ mathrsfs }
\usepackage{multicol}
\usepackage{wrapfig}
\usepackage{stmaryrd}
\usepackage{graphicx}
\usepackage{wrapfig}
\usepackage{amsmath}
\usepackage{amsthm}
\usepackage{amscd}
\usepackage{cite}
\usepackage{tikz}
\usepackage{verbatim}
\usepackage[colorlinks, citecolor=blue]{hyperref}

\DeclareMathOperator {\Ker}{ker}
\DeclareMathOperator {\spann}{span}
\DeclareMathOperator {\dimm}{dim}
\DeclareMathOperator {\Cl}{Cl}
\DeclareMathOperator {\divisor}{div}
\DeclareMathOperator {\WDiv}{WDiv}
\DeclareMathOperator {\Spec}{Spec}

\DeclareMathOperator {\Aut}{Aut}

\DeclareMathOperator {\Reg}{Reg}
\DeclareMathOperator {\Ann}{Ann}
\DeclareMathOperator {\NNA}{(2)}
\DeclareMathOperator {\NA}{(1)}

\DeclareMathOperator {\rank}{rank}

\def\A  {\mathbb A}
\def\N  {\mathbb N}
\def\Z  {\mathbb Z}
\def\C  {\mathbb C}

\def\Ga  {\mathbb G_a}
\def\Q  {\mathbb Q}
\def\K  {\mathbb K}
\def\P  {\mathbb P}
\def\epsilon{\varepsilon}

\let\emptyset\varnothing
\newtheorem{lmm}{Lemma}
\newtheorem{example}{Example}
\newtheorem{corollary}{Corollary}
\newtheorem{theorem}{Theorem}
\newtheorem*{thm*}{Theorem}

\let\oldexmp\example
\renewcommand{\example}{\oldexmp\normalfont}

\newtheorem{oldstm}{Proposition}

\theoremstyle{definition}
\newtheorem{defn}{Definition}

\begin{document}\sloppy

\title{On uniqueness of additive actions on complete toric varieties}
\author{Sergey Dzhunusov}
\address{National Research University Higher School of Economics, Faculty of Computer Science, Pokrovsky boulevard 11, Moscow, 109028 Russia}
\email{dzhunusov398@gmail.com}
\renewcommand{\address}{{
    \footnotesize
    \textsc{Moscow}\par\nopagebreak
    \textit{E-mail} : \texttt{dzhunusov398@gmail.com}
  }}

\date{}
\thanks{The author was supported by RSF grant 19-11-00172.}
\begin{abstract}
  By an additive action on an algebraic variety~$X$ we mean a regular effective action~$\mathbb{G}_a^n\times X\to X$ with an open orbit of the commutative unipotent group~$\mathbb{G}_a^n$.
  In this paper, we give a uniqueness criterion for additive action on a complete toric variety.

\end{abstract}


\maketitle

\section{Introduction}
Let $\K$ be an algebraically closed field of characteristic zero.
Denote its additive group by $\Ga=(\K,+)$.
Consider the commutative unipotent group~$\Ga^n=\Ga\times\ldots\times\Ga$ ($n$~times).
By an additive action on an irreducible algebraic variety~$X$ of dimension $n$ we mean an effective regular action~$\Ga^n\times X\to X$ with an open orbit.
If a complete variety~$X$ admits an additive action, we can consider variety~$X$ as an equivariant completion of affine space~$\A^n$ with respect to the group of parallel translations on~$\A^n$. 

A systematic study of additive actions began with the work of Hassett and Tschinkel~\cite{HT}.
They introduced a correspondence between additive actions on the projective space~$\P^n$ and local $(n+1)$-dimensional commutative associative algebras with a unit; see also~\cite[Proposition~5.1]{KL} for a more general result. 
Hassett-Tschinkel correspondence allows to obtain the classification of additive actions on projective space~$\P^n$ for $n \leq 5$; these are precisely the cases when the number of additive actions is finite.

The study of additive actions was originally motivated by Manin's conjecture about the distribution of rational points of bounded height on algebraic varieties, see works of Chambert-Loir and Tschinkel~\cite{CLT1, CLT2}.

There are some classification results for additive actions on various classes of varieties, 
in particular, on flag varieties \cite{A1, Fe, FH, Dev},
singular del Pezzo surfaces \cite{DL},
Hirzebruch surfaces \cite{HT}, and weighted projective planes \cite{ABZ}.

Some results in this direction are  devoted to the uniqueness of additive actions.
In \cite{Sh}, it is proved that an additive action on a smooth nondegenerate projective quadric is unique up to isomorhpism.
Uniqueness of an additive action on a flag variety that is not isomorphic to a projective space is proved indepently and by completely different methods in~\cite{FH} and \cite{Dev}.

The present work concerns the uniqueness of additive actions in the case of toric varieties. This problem was raised in~\cite[Section~6]{AS}.
In~\cite{De}, it is proved that $\Ga$-actions on a toric variety~$X$ normalized by the acting torus~$T$ are in bijection with some special elements in the character lattice of the torus~$T$ called Demazure roots of the corresponding fan~$\Sigma$. 
Let~$\mathcal{R}(X)$ be the Cox ring of the variety $X$.
Cox~\cite{Cox} noted that normalized $\Ga$-actions on a toric variety can be interpreted as certain $\Ga$-subgroups of automorphisms of the ring~$\mathcal{R}(X)$. 
In turn, such subgroups correspond to homogeneous locally nilpotent derivations of this ring.

In~\cite{AR}, all toric varieties admitting an additive action are described in terms of their fans.
It is proved that if a complete toric variety~$X$ admits an additive action, then it admits an additive action normalized by the acting torus.
Moreover, any two normalized additive actions on~$X$ are isomorphic.

In~\cite{S}, all additive actions on a complete toric surface were classified.
It turns out that there are no more than two non-isomorphic additive actions on a complete toric surface, see Section~\ref{example} of this work for more details.

The present paper suggests a criterion of the uniqueness of additive action on complete toric varieties.
Let~$X$ be a complete toric variety with an acting torus~$T$ that admits an additive action.
Let~$M$ be the character lattice of~$T$ and~$N$ be the lattice of one-parameter subgroups of~$T$.
According to~\cite[Theorem~3]{AR}, the rays of the corresponding fan can be ordered in such a way that the primitive vectors on the first $n$ rays form a basis~$p_1, \ldots, p_n$ of the lattice~$N$ and the remaining rays~$p_{n+1}, \ldots, p_m$ lie in the negative octant with respect to this basis.
Let us denote the dual basis of the basis~${p_1, \ldots, p_n}$ by~${p_1^*,\,\ldots, p_n^*}$.
We also consider the set
\[\mathfrak{R}_i=\{e \in M: \langle  p_i, e\rangle =-1 \text{ and } \langle p_j, e\rangle \geq 0 \text{ for } j \neq i \}.\]
The elements of the set $\bigcup\limits_{i=1}^m\mathfrak{R}_i$ are called Demazure roots of the variety~$X$.

The main result of this paper is described in the following theorem.
\begin{thm*}
  Let $X$ be a complete toric variety admitting an additive action. The following conditions are equivalent:
  \begin{enumerate}
  \item any additive action on the variety~$X$ is isomorphic to the normalized additive action;
  \item the set~$\mathfrak R_i$ is equal to $\{-p_i^*\}$ for every $1 \leq i\leq n$.
  \end{enumerate}
\end{thm*}

Specifically, assertion~(2) of the theorem is equivalent to the fact that the dimension of a maximal unipotent subgroup $U$ of the automorphism group~$\Aut(X)$ is equal to~$n$, and the group~$U$ is the only candidate for a $\Ga^n$-group up to conjugation.
In opposite case, we construct two additive actions and prove that these actions are non-isomorphic.

For example, let us consider the case of projective plane~$\P^2$.
The plane~$\P^2$ can be considered as toric variety.
Let us denote the corresponding fan by $\Sigma$.
As before, we denote the primitive vectors on the rays of the fan~$\Sigma$ by $p_1, p_2, p_3$.
The vector $p_3$ is equal to $-p_1-p_2$.
One can compute directly that \[\mathfrak{R}_1=\{-p_1^*, -p_1^*+p_2^*\}, \,\mathfrak{R}_2=\{-p_2^*, -p_2^*+p_1^*\},\,\mathfrak{R}_3=\{p_1^*, p_2^*\}.\]

The automorphism group~$\Aut(\P^2)$ is ${\rm PGL}_3$ and the dimension of a maximal unipotent subgroup is equal to~$3$.
The variety~$\P^2$ does not satisfy assertion~(2) and by the Theorem we get that $\P^2$ admits at least two non-isomorphic additive actions.
In fact, the variety~$\P^2$ admits precisely two non-isomorphic additive actions, see \cite[Proposition~3.2]{HT}.

After presenting some preliminaries on toric varieties and Cox ring (Section~\ref{coxring}) and $\Ga$-actions and Demazure roots (Section~\ref{intrga}), we describe the results of~\cite{AR} (Section~\ref{intrbas}).
In Section~\ref{draa}, we recall some facts on Demazure roots of a toric variety admitting an additive action from~\cite{S}.
In Section~\ref{mainsection}, we prove the main result of the paper.
Finally, in Section~\ref{example} we give some corollaries and examples and discuss the case of toric surfaces.

\smallskip

The author is grateful to his supervisor Ivan Arzhantsev and to Yulia Zaitseva for useful discussions and comments.

\section{Toric varieties and Cox rings}\label{coxring}
In this section, we introduce basic notation of toric geometry,
see \cite{Fu, CLS} for details.

\begin{defn}
  A \emph{toric variety} is a normal variety~$X$ containing a torus~$T\simeq (\K^{\times})^n$ as a Zariski open subset such that the action of~$T$ on itself extends to an action of~$T$ on~$X$.
\end{defn}

Let~$M$ be the character lattice of~$T$ and~$N$ be the lattice of one-parameter subgroups of~$T$. Let~$\langle \cdot\,, \cdot\rangle: N \times M \to \Z$ be the natural pairing between the lattice~$N$ and the lattice~$M$.
It extends to the pairing~$\langle \cdot\,, \cdot\rangle_{\Q}: N_{\Q} \times M_{\Q} \to \Q$  between the vector spaces~${N_{\Q}=N\otimes_{\Z} \Q}$ and~${M_{\Q}=M\otimes_{\Z} \Q}$.
\begin{defn}
  A \emph{fan}~$\Sigma$ in the vector space~$N_{\Q}$ is a finite collection of strongly convex polyhedral cones~$\sigma$ such that
  \begin{enumerate}
  \item for all cones~$\sigma \in \Sigma$, each face of~$\sigma$ is also in~$\Sigma$;
  \item for all cones~$\sigma_1, \sigma_2 \in \Sigma$, the intersection~$\sigma_1 \cap \sigma_2$ is a face of the cones~$\sigma_1$ and~$\sigma_2$.
  \end{enumerate}
\end{defn}

There is a correspondence between toric varieties~$X$ and fans~$\Sigma$ in the vector space~$N_{\Q}$,
see \cite[Section~3.1]{CLS} for details. 

\smallskip

Here, we recall basic notions of the Cox construction, see~\cite[Chapter~1]{ADHL} for more details.
Let $X$ be a normal variety.
Suppose that the variety $X$ has a free finitely generated divisor class group~$\Cl(X)$ and there are only constant invertible regular functions on~$X$.
Denote the group of Weil divisors on $X$ by $\WDiv(X)$ and consider a subgroup~${K \subseteq \WDiv(X)}$ which maps onto~$\Cl(X)$ isomorphically. The \emph{Cox~ring} of the variety~$X$ is defined as
\[R(X)=\bigoplus_{D\in K} H^0(X,D) \text{, \,where }H^0(X,D)=\{f\in\K(X)^{\times}\mid \divisor(f)+D\geqslant0\}\cup\{0\}\]
and the multiplication on homogeneous components coincides with the multiplication in the field of rational functions~$\K(X)$ and extends to the Cox ring~$R(X)$ by linearity.
It is easy to see that up to isomorphism the graded ring $R(X)$ does not depend on the choice of the subgroup~$K$.

Suppose that the Cox ring~$R(X)$ is finitely generated.
Then ${\overline{X}:=\Spec R(X)}$ is a normal affine variety with an action of the torus~$H_X := \Spec\K[\Cl(X)]$.
There is an open~$H_X$-invariant subset $\widehat{X}\subseteq \overline{X}$ such that the complement~$\overline{X}\backslash\widehat{X}$ is of codimension at least two in $\overline{X}$,
there exists a good quotient~$\pi_X\colon\widehat{X}\rightarrow\widehat{X}/\!/H_{X}$, and the quotient space~$\widehat{X}/\!/H_{X}$ is isomorphic to $X$, see \cite[Construction~1.6.3.1]{ADHL}. Thus, we have the following diagram:
\[
\begin{CD}
\widehat{X} @>{i}>> \overline{X}=\Spec R(X)\\
@VV{/\!/H_{X}}V  \\
X
\end{CD}
\]

It is proved in \cite{Cox} that if~$X$ is toric, then~$R(X)$ is a polynomial algebra~$\K[x_1,\ldots,x_m]$, where the variables~$x_i$ correspond to~$T$-invariant prime divisors~$D_i$ on $X$ or, equivalently, to the rays~$\rho_i$ of the corresponding fan~$\Sigma$.
The $\Cl(X)$-grading on~$R(X)$ is given by~$\deg(x_i)=[D_i]$.
In this case, $\overline{X}$ is isomorphic to $\K^m$, and $\overline{X}\setminus\widehat{X}$ is a union of some coordinate subspaces in $\K^m$ of codimension at least two.
Denote the torus~$(\K^{*})^m$ acting diagonally on the variety~$\overline{X}$ by $\mathbb T$.
Therefore, there are two gradings on ${R}(X)$, namely, $\Z^m$-grading which corresponds to the $\mathbb T$-action and $\Cl(X)$-grading which corresponds to $H_X$-action.

Let us desribe a connection between the gradings by the group~$\Cl(X)$ and by the group~$\Z^m$ on ${R}(X)$.
Each~$w \in M$ gives a character~${\chi^w : T \to \K^{*}}$, and hence, $\chi^w$ is a rational function on~$X$.
By \cite[Theorem 4.1.3]{CLS}, the function~$\chi^w$ defines a principal divisor~${\divisor(\chi^w) = - \sum_{\rho} \langle p_\rho,w\rangle D_\rho}$.
Let us consider a map~${M \longrightarrow \Z^{m}}$ defined by~$w \mapsto  (\langle p_1, w\rangle, \ldots, \langle p_m, w\rangle)$, where $\rho_1, \ldots, \rho_m$ are one-dimensional cones of $\Sigma$ and $p_i$ are primitive vectors on rays $\rho_i$.
By \cite[\S3.4]{Fu}, this map gives an exact sequence
  \[
    0  \longrightarrow  M  \longrightarrow  \Z^m
    \longrightarrow  \Cl(X)  \longrightarrow 0.
  \]
  Here, a divisor $D \in \Z^{\Sigma(1)}=\Z^m$ determines an element $[D] \in \Cl(X)$.

\section{Demazure roots and locally nilpotent derivations}\label{intrga}
Let~$X$ be a toric variety of dimension $n$, and $\Sigma$ be the fan of the variety $X$.
Let~${\Sigma(1)=\{\rho_1, \ldots, \rho_m\}}$ in $N$ be the set of rays of the fan $\Sigma$ and $p_i$ be the primitive lattice vector on the ray $\rho_i$.

For any ray $\rho_i\in \Sigma(1)$, we consider the set $\mathfrak R_i$ of all vectors $e \in M$ such that
\begin{enumerate}
\item $\langle p_i, e\rangle=-1$ and $\langle p_j, e \rangle \geq 0$ for $j \neq i$, $1 \leq j \leq n$;
\item if  $\sigma$ is a cone of $\Sigma$ and $\langle v, e \rangle=0$ for all $v \in \sigma$, then the cone generated by $\sigma$ and $\rho_i$ is in $\Sigma$ as well.
\end{enumerate}

Elements of the set $\mathfrak R = \bigcup\limits_{i=1}^m \mathfrak R_i$ are called \emph{Demazure roots} of the fan $\Sigma$ (see \cite[Section~3.1]{De} or \cite[Section~3.4]{Oda}).
Let us divide the roots $\mathfrak R$ into two classes:
\begin{equation*}
\mathfrak S = \mathfrak R \cap -\mathfrak R,\quad \mathfrak U = \mathfrak R \setminus \mathfrak S.
\end{equation*}
Roots in $\mathfrak{S}$ and $\mathfrak{U}$ are called \emph{semisimple} and \emph{unipotent}, respectively.
\smallskip

A derivation $\partial$ of an algebra $A$ is said to be \emph{locally nilpotent} if for every $f\in A$, there exists $k\in\N$ such that $\partial^k(f) = 0$.
For any locally nilpotent derivation~$\partial$ on $A$, the map~${\varphi_{\partial}:\Ga\times A\rightarrow A}$, ${\varphi_{\partial}(s,f)=\exp(s\partial)(f)}$ defines a
structure of a rational $\Ga$-algebra on $A$. A derivation $\partial$ on a graded ring $A = \bigoplus\limits_{\omega \in K} A_{\omega}$ is said to be \emph{homogeneous} if it respects the~$K$-grading. If ${f,h\in A\backslash \Ker\partial}$ are homogeneous, then~${\partial(fh)=f\partial(h)+\partial(f)h}$ is homogeneous too, and ${\deg\partial(f)-\deg f}$ is equal to ${\deg\partial(h)-\deg h}$.
Thus, any homogeneous derivation $\partial$ has a well-defined \emph{degree} given
as $\deg\partial=\deg\partial(f)-\deg f$ for any homogeneous element~${f\in A\backslash \Ker\partial}$.

Every locally nilpotent derivation of $\Cl(X)$-degree zero on the Cox ring $R(X)$ induces a regular action~$\Ga\times X\to X$. In fact, any regular $\Ga$-action on $X$ arises this way, see \cite[Section~4]{Cox} and \cite[Theorem~4.2.3.2]{ADHL}.
If a $\Ga$-action on a variety $X$ is normalized by the acting torus $T$, then the lifted
$\Ga$-action on $\overline{X}=\K^m$ is normalized by the diagonal torus $\mathbb{T}$. Conversely,
any $\Ga$-action on $\K^m$ normalized by the torus $\mathbb T$ and commuting
with the subtorus $H_{X}$ induces a $\Ga$-action on $X$. This shows that $\Ga$-actions on $X$
normalized by the torus~$T$ are in bijection with locally nilpotent derivations of the Cox ring
$\K[x_1,\ldots,x_m]$ that are homogeneous with respect to the grading by the lattice~$\Z^m$ and have degree zero with respect to the $\Cl(X)$-grading.

For any element $e\in \mathfrak R_i$, we consider the locally nilpotent derivation $\partial_e = \prod_{j \neq i}x_{j}^{\langle p_j, e \rangle}\frac{\partial}{\partial x_i}$ on the algebra $R(X)$.
This derivation has degree zero with respect to the grading by the group~$\Cl(X)$.
This way one obtains a bijection between Demazure roots in $\mathfrak R$ and locally nilpotent derivations on the ring $R(X)$ which are homogeneous with respect to~$\Z^m$-grading and have degree zero with respect to the $\Cl(X)$-grading.
The latter ones, in turn, are in bijection with $\Ga$-actions on $X$ normalized by the acting torus.
\begin{stm}{\cite[Proposition~4.4]{Cox}}\label{homogeneousmonom}
  There is a one-to-one correspondence
  \[
    \begin{aligned}
      \mathfrak{R}_i &\leftrightarrow \{(x_i,x^D) : \text{$x^D
        \in {R}(X)$ is a
        monomial, $x^D \ne x_i$, $\deg(x^D) = \deg(x_i)$}\}.
    \end{aligned}
  \]
\end{stm}
\begin{corollary}\label{homogeneousbasis}
  If a homogeneous component~$C$ of the Cox ring~${R}(X)$ contains a variable~$x_i$, then the vector space~$C$ is spanned by~$x_i$ and~$\partial_e(x_i)$, where~$e$ runs over~$\mathfrak{R}_i$.
\end{corollary}
\section{Complete toric varieties admitting an additive action}\label{intrbas}
In this section, we shortly present the results of \cite{AR}.
Let $X$ be a toric variety of dimension $n$ admitting an additive action, and $\Sigma$ be the fan of the variety $X$.

Since the variety~$X$ admits an additive action, the variety $X$ contains an open $\Ga^n$-orbit isomorphic to the affine space $\K^n$.
By~\cite[Lemma~1]{APS}, any invertible function on the variety~$X$ is constant and the divisor class group $\Cl(X)$ is freely generated.
In particular, the Cox ring~$R(X)$ introduced in Section~\ref{coxring} is well defined.

We denote primitive vectors on the rays of the fan $\Sigma$ by~$p_i$, where $1 \leq i \leq m$.

\begin{defn} \label{completecollection}
A set $e_1,\ldots,e_n$ of Demazure roots of a fan $\Sigma$ of dimension $n$ is called a {\it complete collection} if $\langle p_i,e_j\rangle=-\delta_{ij}$, where $1\le i,j\le n$ for some ordering of $p_1, \ldots, p_m$.
\end{defn}
An additive action on a toric variety $X$ is said to be \emph{normalized} if the image of the group $\Ga^n$ in $\Aut(X)$ is normalized by the acting torus~$T$.
\begin{theorem}{\cite[Theorem 1]{AR}} \label{cc}
  Let $X$ be a toric variety. Then normalized additive actions on $X$ are in bijection with complete collections of Demazure roots of the fan $\Sigma$.
\end{theorem}

\begin{corollary}
A toric variety $X$ admits a normalized additive action if and only if there is a complete collection of Demazure roots of the fan $\Sigma$.
\end{corollary}

\begin{theorem} {\cite[Theorem~2]{AR}}
Any two normalized additive actions on a toric variety are isomorphic.
\end{theorem}

\begin{theorem}{\cite[Theorem 3]{AR}}\label{3con}
Let $X$ be a complete toric variety. The following conditions are equivalent:
\begin{itemize}
\item[(1)]
  there exists an additive action on $X$;
\item[(2)]
  there exists a normalized additive action on $X$;
\item[(3)]
  a maximal unipotent subgroup $U$ of the automorphism group $\Aut(X)$ acts on $X$ with an open orbit.
\end{itemize}
\end{theorem}

\begin{defn}
  The \emph{negative octant} of the rational vector space $V$ with respect to a basis~${f_1, \ldots, f_n}$ is the cone~${\left\{\sum\limits_{i=1}^n \lambda_i f_i \mid \lambda_i \leq 0\right\} \subset V}$.
\end{defn}
\begin{stm}{\cite[Proposition 1]{S}}\label{ort}
  Let $X$ be a complete toric variety.
  The following statements are equivalent:
  \begin{enumerate}
  \item there exists an additive action on $X$;
  \item we can order rays of the fan $\Sigma$ in such a way that the primitive vectors on the first $n$ rays form a basis of the lattice $N$, and the remaining rays lie in the negative octant with respect to this basis.
  \end{enumerate}
\end{stm}
  We can order $p_i$ in such a way that the first $n$ vectors form a basis of the lattice $N$ and the remaining vectors $p_j$ $(n < j \leq m)$ are equal to~$\sum_{i=1}^n-\alpha_{ji} p_i$ for some non-negative integers $\alpha_{ji}$.
\begin{corollary}\label{clbasis}

  The elements $\deg(x_j), n<j\leq m$ form a basis of $\Cl(X)\simeq \Z^{m-n}$
  and
  an element $\deg(x_i), 1\leq i \leq n$ is equal to $\sum\limits_{j=n+1}^m \alpha_{ji}\deg(x_j)$.
\end{corollary}
\begin{proof}
  The matrix of the linear map $M\rightarrow \Z^m$   in the basis~$p_1^*, \ldots, p_n^*$  in $M$ and in the standart basis of the lattice~${\Z^m}$ is equal to
  $
    \left(
      \begin{array}{c}
      I_{n}\\
      -A\\
      \end{array}
    \right)
    $,
  where $I_n$ is the identity matrix of size~$n$ and $A=(\alpha_{ji})$, $n< j\leq m, 1\leq i\leq n$.
Therefore, the elements $\deg(x_j), n<j\leq m$ form a basis of $\Cl(X)\simeq \Z^{m-n}$
  and
  the elements $\deg(x_i)$ are equal to $\sum\limits_{j=n+1}^m \alpha_{ji}\deg(x_j)$.
  
\end{proof}
\section{Demazure roots of a variety admitting an additive action}\label{draa}
Let $X$ be a complete toric variety of dimension $n$ admitting an additive action, and $\Sigma$ be the fan of the variety $X$.
Denote the primitive vectors on the rays $\rho_i$ of the fan $\Sigma$ by~$p_i$, where $1 \leq i \leq m$.

From Proposition~\ref{ort} it follows that we can order $p_i$ in such a way that the first $n$ vectors form a basis of the lattice $N$ and the remaining vectors $p_j$ $(n < j \leq m)$ are equal to~$\sum_{i=1}^n-\alpha_{ji} p_i$ for some non-negative integers $\alpha_{ji}$.
Let us denote the dual basis of the basis $p_1, \ldots, p_n$ by $p_1^*,\,\ldots, p_n^*$.
\begin{lmm}{\cite[Lemma~2]{S}}\label{firstroot}
  Consider $1\leq i\leq n$.
  The set ${\mathfrak R}_i$ is a subset of the set~${-p_i^* + \sum\limits_{l=1,l\neq i}^n\Z_{\geq 0}p_j^*}$ and the vector $-p_i^*$ is contained in~${\mathfrak R}_i$.
 \end{lmm}
Consider the set $\Reg(\mathfrak S) =\{u \in N : \langle u, e \rangle \neq 0 \text{ for all } e \in \mathfrak S\}$.
Any element $u$ from the set~$\Reg(\mathfrak S)$ divides the set of semisimple roots $\mathfrak S$ into two classes as follows:
\[{\mathfrak S_{u}^+ =\{e \in \mathfrak S: \langle u, e\rangle > 0\}},\quad {\mathfrak S_{u}^- =\{e \in \mathfrak S: \langle u, e\rangle < 0\}}.\]
At this point, any element of ${\mathfrak S_{u}}^+$ is called \emph{positive} and any element of $\mathfrak S_{u}^-$ is called \emph{negative}.

\begin{stm}{\cite[Proposition 2]{S}}\label{selective}
  Let $X$ be a complete toric variety admitting an additive action, and ${\mathfrak R = \bigcup\limits_{i=1}^{m} \mathfrak R_i}$ be the set of its Demazure roots.
  Then
  \begin{enumerate}
  \item any element $e \in  \mathfrak R_j, j > n$, is equal to $p_{i'}^*$ for some $1 \leq i'\leq n$;
  \item all unipotent Demazure roots lie in the set $\bigcup\limits_{i=1}^{n}\mathfrak R_i$;
  \item there exists a vector $u\in \Reg(\mathfrak S)$ such that $\mathfrak S_{u}^+ \subset \bigcup\limits_{i=1}^n \mathfrak R_i$.
  \end{enumerate}
\end{stm}

Now we recall basic definitions from the theory of partially ordered sets.
\begin{defn}
  Consider a set $P$ and a binary relation $\leq$ on $P$. Then $\leq$ is a \emph{preorder} if it is reflexive and transitive; i.e., for all $a, b$ and $c$ in $P$, we have:
  \begin{enumerate}
  \item $a \leq a$ (reflexivity);
  \item    if $a \leq b$ and $b \leq c$, then $a \leq c$ (transitivity).
  \end{enumerate}
  Two elements $a, b$ are \emph{comparable} if $a \leq b$ or $b \leq a$.
  Otherwise, they are \emph{incomparable}.
  If every pair of different elements is incomparable, then the preorder is called \emph{trivial}.

  An element $a$ in $P$ is \emph{maximal} if for any element $b$ in $P$ either $b\leq a$ or the elements~$a,b$ are incomparable.
\end{defn}
Define a preorder $\leq$ on the set of rays~${\{\rho_1, \ldots, \rho_n\}}$ in the following way:
\[{\rho_{i_1} \leq \rho_{i_2}}\text{ if } {\alpha_{ji_1}\leq \alpha_{ji_2}}\text{ for  every } {n < j\leq m}.\]
\section{Main results}\label{mainsection}
Let $X$ be a complete toric variety of dimension $n$ admitting an additive action, and $\Sigma$ be the fan of the variety $X$.
Denote the primitive vectors on the rays $\rho_i$ of the fan $\Sigma$ by~$p_i$, where $1 \leq i \leq m$.
From Proposition~\ref{ort} it follows that we can order $p_i$ in such a way that the first $n$ vectors form a basis of the lattice $N$ and the remaining vectors $p_j$ $(n < j \leq m)$ are equal to~$\sum_{i=1}^n-\alpha_{ji} p_i$ for some non-negative integers $\alpha_{ji}$.

  Fix a vector $u \in \Reg(\mathfrak S)$ that satisfies assertion~(3) of Proposition \ref{selective}.
  Hereafter, we write~$\mathfrak S^+$ instead of $\mathfrak S^+_{u}$.
  Denote the set $\mathfrak S^+ \cup \mathfrak U$ by $\mathfrak R^+$.
  From Proposition~\ref{selective}, it follows that the set~$\mathfrak R^+$ lies in the set $\bigcup\limits_{i=1}^n\mathfrak R_i$.
  The one-parameter subgroups of roots from $\mathfrak R^+$ generate the maximal unipotent subgroup~$U$ in the group~$\Aut(X)$ and $\dimm U = |\mathfrak{R}^+|$, see~\cite[Proposition~4.3]{Cox}.
Denote the set ${\mathfrak R^+ \cap \mathfrak R_i}$ by $\mathfrak R^+_i$.

Let us denote a locally nilpotent derivation that corresponds to the Demazure root~${e\in \mathfrak R}$ by $\partial_{e}$.

\begin{theorem}\label{maintheorem}
  Let $X$ be a complete toric variety admitting an additive action.
  The following conditions are equivalent:
  \begin{enumerate}
  \item the set $\mathfrak R_i$ is equal to $\{-p_i^*\}$ for every $1 \leq i\leq n$;
  \item the set $\mathfrak R^+$ is equal to $\{-p_1^*, \ldots, -p_n^*\}$;
  \item the preorder $\leq$ on the set of rays  $\{\rho_1, \ldots, \rho_n\}$ is trivial;
  \item any additive action on variety~$X$ is isomorphic to the normalized additive action.
  \end{enumerate}
\end{theorem}
\begin{proof}
  Equivalence $(1) \Leftrightarrow (2)$ follows from Proposition~\ref{selective}.

  \smallskip

  \begin{lmm}\label{compare}
    The vector $-p_{i_1}^*+p_{i_2}^*$ is a Demazure root if and only if $\rho_{i_1}\geq \rho_{i_2}$.
  \end{lmm}
  \begin{proof}
    The element ${-p_{i_1}^*+p_{i_2}^*}$ is a Demazure root if and only if the element satisfies inequalities~${\langle p_j, -p_{i_1}^*+p_{i_2}^* \rangle \geq 0}$ for all~$n < j \leq m$ since~${\langle p_i, -p_{i_1}^*+p_{i_2}^* \rangle \geq 0}$ for~${i \in \{1, \ldots, n\}\setminus \{i_1\}}$ and ${\langle p_{i_1}, -p_{i_1}^*+p_{i_2}^* \rangle = -1}$.
    The properties
    \[\langle p_j, -p_{i_1}^*+p_{i_2}^* \rangle = \langle -\sum\limits_{i=1}^n\alpha_{ji}p_i, -p_{i_1}^*+p_{i_2}^* \rangle = \alpha_{ji_1}-\alpha_{ji_2}\geq 0,\]
    for $n < j \leq m$ are equivalent to the properties~${\alpha_{ji_1} \geq \alpha_{ji_2}}$ for all $n < j \leq m$, or to the property~${\rho_{i_1}\geq \rho_{i_2}}$.
  \end{proof}

  Let us prove  implication $(1) \Rightarrow (3)$.
  Suppose the converse that $\rho_{i_1} \geq \rho_{i_2}$ for some~${i_1 \neq i_2}$.
  By Lemma~\ref{compare}, the vector $-p^*_{i_1}+p^*_{i_2}$ is a Demazure root and it lies in $\mathfrak R_{i_1}$, a contradiction.

  \smallskip

  \begin{lmm}\label{simple}
    Let~$e$ be a Demazure root from the set~${\mathfrak R_i}$ and ${e \neq -p_i^*}$.
    Then there exists a Demazure root $e' \in \mathfrak R_i$ with $e'=-p_i^*+p_r^*$ for some $1\leq r \leq n$.
    Moreover, if $\langle p_r, e \rangle > 0$ for some~$r$, then the vector $-p_i^*+p_r^*$ is a Demazure root.
  \end{lmm}
  \begin{proof}
    Let~${e=-p_i^*+\sum\limits_{l=1,l\neq i}^n \epsilon_l p_l^*}$, where ${\epsilon_l = \langle p_l, e \rangle\geq 0}$.
    There exists an index~${r\neq i}$ such that~${\epsilon_r\neq 0}$.
    Let us define  a vector~${e'=-p_i^*+p_r^*}$.
    We have ${\langle p_j, e'\rangle \geq \langle p_j, e \rangle \geq 0}$ for all ${n < j \leq m}$. Thus, the element~$e'$ is a Demazure root.

  \end{proof}
  Let us prove  implication $(3) \Rightarrow (1)$.
  Let us assume the converse.
  By Lemma~\ref{simple}, if the set~$\mathfrak{R}_i$ is not equal to $\{-p_i^*\}$, then there exists $r$ such that $-p_i^*+p_r^* \in \mathfrak{R}_i$.
  By Lemma~\ref{compare}, we get $\rho_{i}\geq \rho_{r}$, a contradiction.
  \smallskip
  
  Now we prove  implication $(2) \Rightarrow (4)$.
  A maximal unipotent group~$U$ has dimension~$n$.
  So, the subgroup~$U$ is the only candidate for $\Ga^n$ up to conjugation.

  \smallskip
  
  Let us prove  implication $(4) \Rightarrow (3)$.
  Without loss of generality, let us assume that there exist  rays $\rho_1, \rho_2$ such that $\rho_2 \leq \rho_1$, where $\rho_1$ is a maximal ray.
  By Lemma~\ref{compare}, the vector~$-p_1^*+p_2^*$ is a Demazure root.
  Let us consider the number~${d=\max\{\epsilon : -p_1^*+\epsilon p_2^* \in \mathfrak R_1\}}$
  and take two ordered tuples of derivations:
  \begin{enumerate}
  \item[] $D^{\NA}=(D^{\NA}_1, \ldots, D^{\NA}_n) = (\partial_{-p_1^*}, \partial_{-p_2^*}, \partial_{-p_3^*},\ldots, \partial_{-p_n^*})$;
  \item[] $D^{\NNA}=(D^{\NNA}_1, \ldots, D^{\NNA}_n) =(\partial_{-p_1^*}, \partial_{-p_2^*}+\partial_{-p_1^*+dp_2^*}, \partial_{-p_3^*}\ldots, \partial_{-p_n^*})$.
  \end{enumerate}

  Our goal is to show that these tuples correspond to non-isomorphic additive actions.
  To prove this fact, we find some invariant varieties~$S_{V^{(q)}}(\mathcal C), q=1,2,$ for the above mentioned additive actions and prove that these invariants are non-isomorphic.
  The variety~$S_{V^{(q)}}(\mathcal C), q=1,2,$ is a subset of Cox ring~$R(X)$  connected with an additive action.

  \smallskip
  
  Firstly, we prove that the tuples $D^{\NA}$ and $D^{\NNA}$ correspond to additive actions.
  The derivation $\partial_{-p_2^*}+\partial_{-p_1^*+dp_2^*}$ is a sum of two locally nilpotent derivations of degree zero with respect to $\Cl(X)$-grading.
  Therefore, any derivation in the tuples $D^{(q)}, q=1,2,$ is a derivation of degree zero with respect to the $\Cl(X)$-grading.
  \begin{lmm}
    Derivations in the tuples $D^{\NA}$ and $D^{\NNA}$ are locally nilpotent.
  \end{lmm}
  \begin{proof}
    For any $1 \leq i \leq n$, the derivation $\partial_{-p_i^*}$ is locally nilpotent since it corresponds to a Demazure root.
    We should check that the derivation $\partial_{-p_2^*}+\partial_{-p_1^*+dp_2^*}$ is locally nilpotent.
    It easily follows from the following:
    \begin{align*}
        (\partial_{-p_2^*}+\partial_{-p_1^*+dp_2^*})(x_1)&\in \K[x_2, x_{n+1}, x_{n+2}, \ldots, x_m];\\
        (\partial_{-p_2^*}+\partial_{-p_1^*+dp_2^*})(x_2)&\in \K[x_{n+1}, x_{n+2}, \ldots, x_m];\\
      (\partial_{-p_2^*}+\partial_{-p_1^*+dp_2^*})(x_j)&=0 \text{, for } 2 < j \leq m.  
    \end{align*}

  \end{proof}
  \begin{lmm}
    Derivations in the tuple $D^{(q)}, q=1,2,$ pairwise commute.
  \end{lmm}
  \begin{proof}
    From Theorem~\ref{cc} we know that derivations in the tuple corresponding to the normalized additive action commute, as a result~${[\partial_{-p_i^*}, \partial_{-p_j^*}]=0}$.
    It remains to check that~${[\partial_{-p_2^*}+\partial_{-p_1^*+dp_2^*}, \partial_{-p_i^*}]=[\partial_{-p_1^*+dp_2^*}, \partial_{-p_i^*}]=0}$ if~${i \neq 2}$.
    This can be checked directly.
  \end{proof}
    By these lemmas, we get that the ordered tuples~$D^{(q)}, q=1,2,$ correspond to  actions~$a^{(q)}$ on the variety~$X$ by the group~$\Ga^n$.

  \begin{defn}
    Let us call an ordered tuple of locally nilpotent derivations~$D=(D_1, \ldots, D_n)$ \emph{triangular} if $D_ix_i \neq 0$ and $D_l x_i =0$ if $i > l$.
  \end{defn}
  It is easy to check that the tuples of derivations $D^{\NA}$ and $D^{\NNA}$ are triangular.
  \begin{lmm}
    The $\Ga^n$-action corresponding to a triangular tuple of commuting locally nilpotent derivations has an open orbit on the variety~$X$.
    Thus, a triangular ordered tuple of locally nilpotent derivations defines an additive action $\Ga^n \times X \to X$.
  \end{lmm}
  \begin{proof}
    
    We prove that there exists a point~${p =(x_1, \ldots, x_m)\in \widehat{X}\subset \overline{X}}$ such that ${\dimm (\Ga \times H_X) p=m}$.
    The Jacobian of the orbit morphism $\varphi_p \colon \Ga \times H_X \to \widehat{X}$ at the identity of the group~$\Ga\times H_X$ is equal to $\prod\limits_{i=1}^nD_ix_i\prod\limits_{j=n+1}^mx_j$.
    There exists a point~${p\in\widehat{X}}$, where the product~${\prod\limits_{i=1}^nD_ix_i\prod\limits_{j=n+1}^mx_j}$ is not zero.
    The dimension of the tangent space of the orbit~$(\Ga^n\times H_X)p$ at the point $p$ is equal to $\dim \overline{X}=m$.
    Thus, the orbit~$(\Ga^n\times H_X)p$ on the variety~$\overline{X}$ is open.
    Consequently, after factorization~$\pi_X: \widehat{X} \to X$ the orbit $\Ga^n\pi_X(p)$ is open on the variety~$X$ as well.
  \end{proof}
      Therefore, the action~$a^{(q)}, q=1,2,$ is an additive action.

      \smallskip
      
  Now we prove that actions corresponding to the tuples $D^{\NA}$ and $D^{\NNA}$ are non-isomorphic.
  Let us consider an equivalence relation on the set of rays~$\Sigma(1)$ determined by
    \[\rho_{i_1} \sim \rho_{i_2} \Longleftrightarrow \deg(x_{i_1}) = \deg(x_{i_2}) \text{ in } \Cl(X).\]
    This partitions $\Sigma(1)$ into disjoint subsets $\bigsqcup\limits_{i=1}^r\Sigma(1)_{i}$, where each subset $\Sigma(1)_i$ corresponds to a set of variables of the same degree $\omega_i$.
    Let $\mathcal{C}_i = \{f \in \mathcal R(X) : \deg(f) = \omega_i\}$ be the homogeneous component.
    Let us consider the vector space~$\mathcal{C}_i$ as an algebraic variety~$\A^{\dimm \mathcal{C}_i}$.
    We take the algebraic variety~$\mathcal{C} = \bigcup\limits_{i=1}^r \mathcal{C}_i$.
  We consider two vector spaces~${V^{\NA} = \{\sum\limits_{i=1}^n s_i D^{\NA}_i : s_i \in \K\}}$ and ${V^{\NNA} = \{\sum\limits_{i=1}^n s_i D^{\NNA}_i: s_i \in \K\}}$.
  For every element~${f \in \mathcal{C}}$, we regard the subspace~${\Ann_V f =\{v \in V:vf=0 \}}$
  of a space~$V$ of derivations.
  
  Let us consider the following sets:
  \[S_V(\mathcal{C}_i) = \{f\in \mathcal{C}_i : \dimm \Ann_V f \geq  \dimm V-1\},\]
  \[S_V(\mathcal{C}) = \{f\in \mathcal{C} : \dimm \Ann_V f \geq \dimm V-1\} = \bigcup\limits_{i=1}^r S_V(\mathcal{C}_i)\]
    \begin{lmm}
    The subset $S_V(\mathcal{C}_i)$ is a closed subvariety of the variety $\mathcal{C}_i$.
  \end{lmm}
  \begin{proof}
    We have the system of linear equations $vf = 0$, where $v \in V$ and $f \in \mathcal{C}_i$ is a certain fixed element.
    We  choose some bases  in~$V$ and in~$\mathcal{C}_i$.
    In these terms, the condition $\dim \Ann_f V \geq \dim V - 1$ means that the matrix of system of linear equations $vf = 0$ has rank less than 2 or, equivalently, every~$2\times 2$ submatrix is singular.
    Thus, $S_V(\mathcal{C}_i)$ is the subvariety of $\mathcal{C}_i$ defined by  equations~${\det(M)=0}$, where $M$ runs over all $2\times 2$ submatrices of the matrix of linear equation.
  \end{proof}
    By \cite[Theorem~3.2.6]{CLS} $T$-invariant divisors $D_1, \ldots, D_m$ on the variery~$X$ as well as the elements~$[D_1], \ldots, [D_m]\in \Cl(X)$ are canonical.
  Therefore, the degrees of the variables are canonical, since the degrees are equal to $[D_1], \ldots, [D_m]$.
  As a result if additive actions~${a^{(1)}, a^{(2)}}$ are isomorphic,
  then the varieties~${S_{V^{(1)}}(\mathcal{C}), S_{V^{(2)}}(\mathcal{C})}$ should be isomorphic.
  We are going to prove that varieties $S_{V^{\NA}}(\mathcal{C})$ and $S_{V^{\NNA}}(\mathcal{C})$ are not isomorphic.

  Without loss of generality, we suppose $x_1 \in \mathcal{C}_1$.
  
  \begin{lmm}
    For $i \neq 1$, we have $S_{V^{\NA}}(\mathcal{C}_i) = S_{V^{\NNA}}(\mathcal{C}_i)$.
  \end{lmm}
  \begin{proof}
    
    We prove that the derivation~$\partial_{-p_1^*+dp_2^*}$ is zero on the vector space~${\mathcal C_i, i>1}$.
    Assume the converse.
    We know that~${\partial_{-p_1^*+dp_2^*} = f\frac{\partial}{\partial x_1}}$, where~${f \in {R}(X)}$.
    It follows that the derivation~$\frac{\partial}{\partial x_1}$ is not zero on the vector space~$\mathcal{C}_i$.
    There exists a certain variable~${x_{l} \in \mathcal{C}_i}$, ${l \neq 1}$.
    By Corollary~\ref{homogeneousbasis}, we get~${\mathcal{C}_i = \{\lambda x_l + \sum_{e \in \mathfrak{R}_l}\lambda_e \partial_e(x_l) : \lambda, \lambda_e \in \K\}.}$
    Since $l \neq 1$ we obtain $\frac{\partial}{\partial x_1}(x_l) = 0$.
    Also, from the definition of Demazure root we get~${\partial_e(x_l) = x_1^{\langle p_1, e \rangle}g}$,
    ${g \in \K[x_2, \ldots, x_m]}$.
    Since the ray~$\rho_1$ is maximal, by Lemma~\ref{compare} no vector~$-p_l^* + p_1^*$ is a Demazure root.
    Then by Lemma~\ref{simple} the pairing~$\langle p_1, e \rangle$ is equal to zero
    and~$\frac{\partial}{\partial x_1}(\partial_e(x_l))=0$, a contradiction.

    As the derivation~$\partial_{-p_1^*+dp_2^*}$ is zero,   the tuples of derivations~$D^{\NA}$ and $D^{\NNA}$ are equal.
  \end{proof}
  By Corollary~\ref{homogeneousbasis}, for every element $f\in \mathcal{C}_1$, we can consider a representation~${f = \lambda x_1 + \sum\limits_{e\in \mathfrak R_1} \lambda_e \partial_e(x_1)}$ in the basis $x_1, \partial_e(x_1)$, where $e \in \mathfrak R_1$.

  Since $\partial_{-p_i^*}=\prod_{l=n+1}^mx_l^{\alpha_{il}}\frac{\partial}{\partial x_i},\, 1\leq i \leq n,$ and $\partial_{-p_1^*+dp_2^*}=x_2^d\prod_{l=n+1}^mx_l^{\alpha_{1l}-d\alpha_{2l}}\frac{\partial}{\partial x_1}$,
 the image $D_i^{(q)}(\lambda x_1+\sum_{e\in \mathfrak{R}_1}\lambda_e\partial_e(x_1)), 1\leq i \leq n$ and $q=1,2,$ belongs to $\spann_{e\in \mathfrak{R}_1}{\langle\partial_e(x_1)\rangle}$.
Let us introduce the coefficients~$\upsilon_{e,i}^{(q)}$:
  \[D_i^{(q)}\left(\lambda x_1+\sum_{e\in \mathfrak{R}_1}\lambda_e\partial_e(x_1)\right)=\sum_{e\in\mathfrak{R}_1}\upsilon_{e,i}^{(q)}\partial_e(x_1).\]
  \begin{lmm}\label{closed}
    The algebraic variety~$S_{V^{\NNA}}(\mathcal{C}_1)$ is the proper closed subset of the variety~$S_{V^{\NA}}(\mathcal{C}_1)$.
  \end{lmm}
  \begin{proof}
    We prove that $S_{V^{\NNA}}(\mathcal{C}_1)\subset\{\lambda=0\}$.
    For this, we choose ${2\times 2}$ submatrix~${L=\left(\begin{matrix}\upsilon^{\NNA}_{-p_1^*,1}&\upsilon^{\NNA}_{-p_1^*,2}\\\upsilon^{\NNA}_{-p_1^*+dp_2^*,1}&\upsilon^{\NNA}_{-p_1^*+dp_2^*,2}\end{matrix}\right)}$.
    We have
    \[\partial_{-p_1^*}(\lambda x_1 + \sum_{e\in \mathfrak{R}_1} \lambda_e \partial_e(x_1))=\lambda\partial_{-p_1^*}(x_1),\]
    \[(\partial_{-p_2^*}+\partial_{-p_1^*+dp_2^*})\left(\lambda x_1 + \sum_{e\in \mathfrak{R}_1} \lambda_e \partial_e(x_1)\right)=\lambda\partial_{-p_1^*+dp_2^*}(x_1)+\kern-0.8em\sum\limits_{\substack {e\in \mathfrak{R}_1\\ e+p_2^*\in \mathfrak{R}_1}}\kern-0.8em\left(\langle p_2, e \rangle+1\right)\lambda_{e+p_2^*}\partial_e(x_1).\]
    Since $d$ is maximal with $-p_2^*+dp_1^*$ being a Demazure root, we have $\upsilon^{\NNA}_{-p_1^*+dp_2^*,2}=\lambda$.
    The submatrix~$L$ is equal to~$\left(\begin{matrix}\lambda&\lambda_{-p_1^*+p_2^*}\\0&\lambda\end{matrix}\right)$.
    Thus, ${S_{V^{\NNA}}(\mathcal{C}_1)\subset\{\lambda=0\}}$.
    We know that~${\spann_{e\in\mathfrak{R}_1} {\partial_e(x_1)}\subset \Ker \partial_{-p_2^*+dp_1^*}}$.
    Therefore, if $\lambda=0$ then the systems of linear equations are the same for tuples $D^{\NA}$ and $D^{\NNA}$.
    This follows that \[S_{V^{\NNA}}(\mathcal{C}_1) = S_{V^{\NNA}}(\mathcal{C}_1)\cap \{\lambda=0\}= S_{V^{\NA}}(\mathcal{C}_1)\cap \{\lambda=0\}.\]

    Let us prove that~${S_{V^{\NA}}(\mathcal{C}_1)\not\subset \{\lambda=0\}}$.
    Since ${\sum s_i D_i^{\NA} (x_1)=s_1\partial_{-p_1^*}(x_1)}$ the point~${\lambda=1}$ and all ${\lambda_e=0}$ belongs to the variety~${S_{V^{\NA}}(\mathcal{C}_1)}$.
  \end{proof}

  By Lemma \ref{closed}, the varieties $S_{V^{\NA}}(\mathcal{C})$ and $S_{V^{\NNA}}(\mathcal{C})$ are not isomorphic.
  This completes the proof of impication $(1)\Rightarrow (3)$.
  So, Theorem~\ref{maintheorem} is proved.
  
\end{proof}

\section{Corollaries and examples}\label{example}
We preserve notation of the previous section.
The next corollary follows from Theorem~\ref{maintheorem}.
\begin{corollary}
  Let $X$ be a complete toric variety admitting an additive action.
  The following conditions are equivalent:
  \begin{enumerate}
  \item the dimension of a maximal unipotent subgroup of the automorphism group~$\Aut(X)$ is equal to the dimension of the variety~$X$;
  \item any additive action on~$X$ is isomorphic to the normalized additive action.
  \end{enumerate}

\end{corollary}
\begin{proof}
The dimension of a maximal unipotent subgroup is equal to the size of the set~$\mathfrak{R}^+$.
\end{proof}
  
\begin{example}
 Let us consider the set of vectors $p_1, \ldots, p_5$ in $N=\Z^2$
  such that the vectors $p_1, p_2$ form a basis of $N$, ${p_{3}=-p_1+p_2}$, ${p_4=-2p_1-p_2}$ and ${p_{5}=-p_1-p_2}$.
  Let  $\rho_1, \ldots, \rho_{5} \subset N_{\Q}$ be the rays generated by the vectors~$p_1, \ldots, p_{5}$, respectively.
  Let us consider a complete toric variety~$X$ with the fan~$\Sigma$ such that~$\Sigma(1) = \{\rho_1, \ldots, \rho_{5}\}$.
  It can be computed directly that $\mathfrak{R}_1=\{-p_1^*, -p_1^*+p_2^*\}$ and $\mathfrak R_i=\emptyset, i\geq 2$.
  Therefore, a maximal unipotent subgroup of the group~$\Aut(X)$ has dimension~$2$, but there is no additive action on the variety~$X$ by Lemma~\ref{firstroot}.

\end{example}

Now let us recall the main result of~\cite{S} and explain the connection between this result and Theorem~\ref{maintheorem}.

\begin{defn}\label{defwide}
  Let us consider a complete two-dimensional fan $\Sigma$ that corresponds to a toric surface admitting an additive action.
  The primitive vectors on the rays in the fan~$\Sigma$ are equal to vectors ${p_1, p_2}$ and ${-\alpha_{j1}p_1-\alpha_{j2}p_2}$, ${2 < j \leq m}$ for some $\alpha_{ji} \geq 0, i=1,2$.
  Let us call a fan $\Sigma$ \emph{wide} if  there exist indices ${2 < j, j' \leq m}$ such that~${\alpha_{j1} > \alpha_{j2}}$ and ${\alpha_{j'1} < \alpha_{j'2}}$.
\end{defn}

\begin{theorem}{\cite[Theorem 3]{S}}\label{surface}
Let $X$ be a complete toric surface  admitting an additive action.
  Then an additive action on $X$ is unique up to isomorphism if and only if the fan $\Sigma$ is wide; otherwise, there exist preciesly two non-isomorphic additive actions, where one is normalized and the other is not.
\end{theorem}
The following corollary of Theorem~\ref{maintheorem} explains a connection between Theorem~\ref{maintheorem} and Theorem~\ref{surface}.
\begin{corollary}
  Let $X$ be a complete toric variety admitting an additive action. The following conditions are equivalent:
  \begin{enumerate}
  \item any additive action is isomorphic to the normalized additive action;
  \item the image under the  projection along the coordinate plane~${\spann\{p_1, \ldots, \widehat{p_{l_1}}, \ldots, \widehat{p_{l_2}}, \ldots, p_n\}}$ of the system of rays $\Sigma(1)$ to the plane spanned vectors $p_{l_1}, p_{l_2}$ determines a wide fan for every~${1 \leq l_1\neq l_2 \leq n}$.

  \end{enumerate}
\end{corollary}
\begin{proof}
  The image of the projection of the fan to the plane spanned by vectors $p_{l_1}, p_{l_2}$ is wide if and only if the rays $\rho_{l_1}$ and $\rho_{l_2}$ are incomparable.
  Thus, the corollary stems from  equivalence~$(3)\Leftrightarrow (4)$ of Theorem~\ref{maintheorem}.
\end{proof}
\begin{corollary}\label{weightedprojective}
  Let $X$ be a complete toric variety admitting an additive action.
  If we have~$m=n+1$ or, equivalently, $\rank \Cl(X)=1$, then there are at least two non-isomorphic additive actions.
\end{corollary}
\begin{proof}
By definition, the preorder on the rays~${\rho_1, \ldots, \rho_n}$ is the same as the natural order on numbers~${\alpha_{n+1, 1}, \ldots, \alpha_{n+1, n}}$.
  Every two elements are comparable.
  Therefore, the preorder is not trivial.
\end{proof}
  Corollary~\ref{weightedprojective} covers the case of weighted projective spaces.
  By \cite[Proposition~2]{AR}, a weighted projective space~${\P(a_0, \ldots, a_n)}, {a_0 \leq a_1 \leq \ldots \leq a_n}$ admits an additive action if and only if~$a_0=1$.
  By this corollary, on a weighted projective space~$\P(1, a_1, \ldots, a_n)$ there are at least two non-isomorphic additive actions.
  \smallskip

  The final example shows that in the case $m=n+2$ an additive action can be unique.
\begin{example}
  Let us consider the set of vectors $p_1, \ldots, p_{n+2}$ in $N=\Z^n$
  such that the vectors $p_1, \ldots, p_n$ form a basis of $N$, ${p_{n+1}=-\sum_{i=1}^nip_i}$ and ${p_{n+2}=-\sum_{i=1}^n(n-i+1)p_i}$.
  Let us consider the rays $\rho_1, \ldots, \rho_{n+2} \subset N_{\Q}$ generated by $p_1, \ldots, p_{n+2}$.
  We consider a complete toric variety~$X$ with a fan $\Sigma$ such that $\Sigma(1) = \{\rho_1, \ldots, \rho_{n+2}\}$.
  By Theorem~\ref{maintheorem} an additive action on such a variety is unique.
\end{example}

\end{document}